\documentclass[12pt,leqno]{article}
\usepackage{amssymb,amsmath,amsthm}
\usepackage[all]{xy}
\usepackage[alphabetic,lite]{amsrefs}
\numberwithin{equation}{section}
\usepackage[top=2cm,bottom=2cm,left=2cm,right=4cm,marginparsep=0.3cm,marginparwidth=3cm,includefoot]{geometry}
\usepackage{mathtools}
\usepackage{enumerate}
\usepackage[colorlinks=true, pdfstartview=FitV, linkcolor=blue, citecolor=blue, urlcolor=blue]{hyperref}

\usepackage{enumerate}
\usepackage{mathrsfs}
\usepackage{accents}
\usepackage{xspace}
\usepackage[normalem]{ulem}  
\usepackage{color}

\newenvironment{bleu}
{\relax\color{blue}}
{\hspace*{.3ex}\relax}
\newcommand{\bev}{\begin{bleu}}
\newcommand{\ev}{\end{bleu}}

\newcommand{\beb}{\begin{bleu}}
\newcommand{\eb}{\end{bleu}}
\newcommand{\spa}{\vspace{0.5ex}\noindent}

\newenvironment{rouge}
{\relax\color{red}}
{\hspace*{.3ex}\relax}

\newcommand{\ber}{\begin{rouge}}
\newcommand{\er}{\end{rouge}}

\newcommand{\nc}{\newcommand}
\newcommand{\define}[1]{\emph{#1}}

\nc{\on}{\operatorname}
\newtheorem{theorem}{Theorem}[section]
\newtheorem{proposition}[theorem]{Proposition}
\newtheorem{lemma}[theorem]{Lemma}
\newtheorem{corollary}[theorem]{Corollary}

\theoremstyle{definition}

\newtheorem{definition}[theorem]{Definition}
\newtheorem{notation}[theorem]{Notation}

\newtheorem{remark}[theorem]{Remark}

\newtheorem{conjecture}[theorem]{Conjecture}

\nc{\Prop}{\begin{proposition}}
\nc{\enprop}{\end{proposition}}
\nc{\Lemma}{\begin{lemma}}
\nc{\enlemma}{\end{lemma}}
\nc{\Cor}{\begin{corollary}}
\nc{\encor}{\end{corollary}}
\nc{\Def}{\begin{definition}}
\nc{\enDef}{\end{definition}}
\nc{\Th}{\begin{theorem}}
\nc{\entheorem}{\end{theorem}}
\nc{\Rem}{\begin{remark}}
\nc{\enrem}{\end{remark}}

\newcommand{\R}{{\mathbb{R}}}

\newcommand{\Z}{{\mathbb{Z}}}

\nc{\forl}{[\mspace{-.3mu}[\hbar]\mspace{-.3mu}]}
\nc{\Ls}{(\mspace{-.3mu}(\hbar)\mspace{-.3mu})}

\newcommand{\cor}{{\bf k}}
\def\phi{{\varphi}}
\def\epsilon{\varepsilon}

\nc{\Proof}{\begin{proof}}
\nc{\QED}{\end{proof}}

\nc{\qtq}[1][{and}]{\quad\text{#1}\quad}
\nc{\qt}[1]{\quad\text{#1}}
\nc{\dist}{\mathrm{dist}}

\nc{\PL}{\ensuremath{\rm PL}\xspace}
\nc{\LPL}{\ensuremath{\rm LPL}\xspace}
\nc{\al}{\alpha}
\nc{\la}{\lambda}
\nc{\eps}{\epsilon}
\renewcommand{\emptyset}{\varnothing}

\def\BBP{{\mathbb P}}

\def\BBV{\mathbb{V}}
\def\BBW{\mathbb{W}}

\newcommand{\stkHom}[1][]{\mathfrak{Hom}_{\raise1.5ex\hbox to.1em{}#1}}

\nc{\bS}{{\bf S}}
\nc{\RR}{\mathrm{R}}
\nc{\LL}{\mathrm{L}}

\def\sha{\mathscr{A}}

\def\shc{\mathscr{C}}
\def\shd{\mathscr{D}}
\def\she{\mathscr{E}}

\def\shg{\mathscr{G}}

\def\sht{\mathscr{T}}

\newcommand{\ol}{\overline}
\newcommand{\bl}{\bigl(}
\newcommand{\br}{\bigr)}

\newcommand{\lp}{{\rm(}}
\newcommand{\rp}{{\rm)}}

\newcommand{\into}{\hookrightarrow}

\newcommand{\Int}{{\rm Int}}

\newcommand{\vvert}{\Vert}

\renewcommand{\to}[1][]{\xrightarrow[]{#1}}

\newcommand{\isoto}[1][]{\xrightarrow[#1]%
{{\raisebox{-.6ex}[0ex][-.6ex]{$\mspace{1mu}\sim\mspace{2mu}$}}}}

\newcommand{\To}[1][]{\xrightarrow[]{\mspace{10mu}{#1}\mspace{10mu}}}

\newcommand{\Hom}[1][]{\mathrm{Hom}_{\raise1.5ex\hbox to.1em{}#1}}
\newcommand{\RHom}[1][]{\RR\mathrm{Hom}_{\raise1.5ex\hbox to.1em{}#1}}
\newcommand{\Ext}[2][]{\mathrm{Ext}_{\raise1.5ex\hbox to.1em{}#1}^{#2}}
\renewcommand{\hom}[1][]{{\mathscr{H}\mspace{-4mu}om}_{\raise1.5ex\hbox to.1em{}#1}}
\newcommand{\rhom}[1][]{{\RR\mathscr{H}\mspace{-3mu}om}_{\raise1.5ex\hbox to.1em{}#1}}
\newcommand{\ext}[2][]{{\mathscr{E}\mspace{-2mu}xt}_{%
\raise1.5ex\hbox to.1em{}#1}^{#2}}

\newcommand{\Tens}[1][]{\mathbin{\otimes_{\raise1.5ex\hbox to-.1em{}{#1}}}}
\newcommand{\LTens}[1][]{\mathbin{\otimes_{\raise1.5ex\hbox to-.1em{}#1}^{L}}}
\newcommand{\Tor}[2][]{\mathrm{Tor}^{\raise1.5ex\hbox to.1em{}#1}_{#2}}
\newcommand{\tens}[1][]{\mathbin{\otimes_{\raise1.5ex\hbox to-.1em{}{#1}}}}

\newcommand{\dtens}[1][]%
{{\overset{\mathrm{L}}{\underline{\otimes}}}_{#1}}

\newcommand{\etens}{\mathbin{\boxtimes}}

\nc{\aut}{{\sha\mspace{-1mu}\mit{ut}}\,}

\newcommand{\shend}{\operatorname{{\she\mspace{-2mu}\mathit{nd}}}}
\newcommand{\Endo}[1][]{\mathrm{End}_{\raise1.5ex\hbox to.1em{}#1}}
\newcommand{\sendo}[1][]{{\shend}_{\raise1.5ex\hbox to.1em{}#1}}

\newcommand{\Aut}[1][]{\mathrm{Aut}_{\raise1.5ex\hbox to.1em{}#1}}

\newcommand{\RC}{{\rm C}}

\newcommand{\Rb}{{\rm b}}

\newcommand{\SSi}{\on{SS}}

\nc{\sdot}[1]{\overset{\raisebox{-.3ex}{\scalebox{.4}{$\bullet$}}}{#1}}

\newcommand{\eim}[1]{{#1}_!}
\newcommand{\roim}[1]{\RR{#1}_*}
\newcommand{\reim}[1]{\RR{#1}_!}
\newcommand{\opb}[1]{#1^{-1}}
\newcommand{\epb}[1]{#1^{!}}

\DeclareMathOperator{\supp}{supp}

\newcommand{\reg}{{\rm reg}}

\newcommand{\Rc}{{\R\rm c}}

\newcommand{\Vg}{{\BBV_\gamma}}

\newcommand{\md[1]}{{\rm Mod}({#1})}

\newcommand{\mdrcg[1]}{{\rm Mod_{\R\rm c,\gamma^{\circ a}}}({#1})}
\newcommand{\mdrcl[1]}{{\rm Mod_{\rm{PL}}}({#1})}
\newcommand{\mdrclg[1]}{{\rm Mod_{\rm{PL,\gammac}}}({#1})}

\newcommand{\mdrc[1]}{{\rm Mod_{\R\rm c}}({#1})}

\newcommand{\mdg[1]}{{\rm Mod_{\gamma^{\circ a}}}({#1})}

\newcommand{\rclg}{\rm{PL},\gammac}
\newcommand{\rcl}{\rm{PL}}
\newcommand{\rcc}{\R\rm c,\rm{c}}
\newcommand{\rcg}{\R\rm c,\gamma^{\circ a}}
\newcommand{\gammac}{\gamma^{\circ a}}
\newcommand{\Barc}{{\bf Bar}_\gamma}

\newcommand{\eqdot}{\mathbin{:=}}

\newcommand{\cl}{\colon}
\newcommand{\scbul}{{\,\raise.4ex\hbox{$\scriptscriptstyle\bullet$}\,}}

\newcommand{\ba}{\begin{array}}
\newcommand{\ea}{\end{array}}

\newcommand{\bnum}{\begin{enumerate}[{\rm(i)}]}
\newcommand{\enum}{\end{enumerate}}
\newcommand{\banum}{\begin{enumerate}[{\rm(a)}]}
\newcommand{\eanum}{\end{enumerate}}
\newcommand{\bna}{\begin{enumerate}[{\rm(a)}]}

\newenvironment{myequation}
{\relax\setlength{\arraycolsep}{1pt}\begin{eqnarray}}
{\end{eqnarray}}
\newenvironment{myequationn}
{\relax\setlength{\arraycolsep}{1pt}\begin{eqnarray*}}
{\end{eqnarray*}}

\nc{\eq}{\begin{myequation}}
\nc{\eneq}{\end{myequation}}
\nc{\eqn}{\begin{myequationn}}
\nc{\eneqn}{\end{myequationn}}

\newcommand{\set}[2]{\left\{#1 \mathbin{;} #2 \right\}}

\nc{\Der}{\on{D}}

\nc{\Derb}{\mathrm{D}^{\mathrm{b}}}
\nc{\rmD}{\mathrm{D}}
\nc{\rb}{\mathrm{b}}
\nc{\hs}{\hspace*}
\nc{\ms}{\mspace}
\nc{\Supp}{\on{Supp}}
\nc{\tr}{\on{tr}}

\newcommand{\RD}{{\rm D}}
\nc{\RDD}{\mathrm{D}^\prime}

\nc{\conv}[1][]{\mathop{\circ}\limits_{#1}}

\nc{\sconv}[1][]{\mathop{\ast}\limits_{#1}}

\nc{\ssum}{\mathop{\mbox{\normalsize$\sum$}}}
\nc{\de}[1][X]{\delta_{#1}} 
\nc{\vs}{\vspace}
\nc{\soplus}{\mathop{\raisebox{.1ex}{\scalebox{.8}{$\displaystyle\bigoplus$}}}}

\nc{\be}{\begin{enumerate}}
\nc{\ee}{\end{enumerate}}
\nc{\stan}{\mathrm{stan}}
\nc{\Db}{\RD^\Rb}
\nc{\pt}{\mathrm{pt}}
\nc{\BBD}{\mathbb{D}}
\nc{\rC}{\mathrm{C}}
\nc{\scup}{\mathop{\text{\scriptsize\raisebox{.5ex}{$\displaystyle\bigcup$}}}}
\nc{\tX}{{\widetilde{X}}}
\nc{\Gr}{\on{Gr}}
\nc{\codim}{\on{codim}}

\DeclareMathOperator{\musupp}{\SSi}
\DeclareMathOperator{\musuppb}{\mu supp}

\DeclarePairedDelimiterX{\innerp}[2]{\langle}{\rangle}{#1,#2}

\newcommand{\phig}{\phi_\gamma}
\newcommand{\rhomc}[1][]
{{\mathscr{H}\mspace{-3mu}om}^\star_{\raise1.5ex\hbox to.1em{}#1}}
\nc{\csum}{\mathbin{\widehat{+}}}
\nc{\La}{\Lambda}
\nc{\bsout}[1]{\ber\sout{#1}\er}

\begin{document}

\title{Piecewise linear sheaves}
\author{Masaki Kashiwara and Pierre Schapira%
}

\date{}
\maketitle

\begin{abstract}
On a finite-dimensional real vector space, we give a microlocal characterization of (derived) piecewise linear sheaves (\PL sheaves) and prove that the triangulated category of such sheaves is generated by sheaves associated with convex polyhedra. We then give a similar theorem for \PL $\gamma$-sheaves, that is, \PL sheaves associated with the $\gamma$-topology, for a closed convex polyhedral  proper cone $\gamma$. Our motivation is that convex polyhedra may be considered as building blocks for higher dimensional barcodes.
 \end{abstract}
{\renewcommand{\thefootnote}{\mbox{}}
\footnote{The research of M.K
was supported by Grant-in-Aid for Scientific Research (B)
15H03608, Japan Society for the Promotion of Science.}
\footnote{The research of P.S was supported by the  ANR-15-CE40-0007 ``MICROLOCAL''.}
\addtocounter{footnote}{-2}
}

\tableofcontents

\section*{Introduction}
Persistent homology is an essential tool of Topological Data Analysis appearing in numerous papers. To our opinion it may be interpreted as follows (see~\cite{KS18}). One has some data on a manifold $X$ which  define   a constructible sheaf $F$ on $X$, one has a function $f\cl X\to\R$ (playing the role of a Morse function) and one can calculate the direct image by $f$ of the data, that is, the derived direct image $\roim{f}F$. Assuming that $f$ is proper on the support of $F$, one gets a constructible (derived) sheaf on $\R$ and a variant of a theorem of  Crawley-Boevey~\cite{CB14} (see also~\cite{Gu16} and,  for the non compact case,~\cite{KS18}*{Th.~2.17}) asserts that such an object is nothing but a graded barcode. Moreover, in practice, the data on $X$ are  associated with an order and it follows that the barcodes are half-closed intervals (e.g., closed on the left and open on the right). In the langage of sheaves, this means that one gets a $\gamma$-sheaf represented by a $\gamma$-barcode, where $\gamma$ is the cone $\{t\in\R;t\leq0\}$.

However it is natural in many problems to replace the ordered set $(\R,\leq)$ with an ordered  finite dimensional vector space $\BBV$ and the order may be deduced from the data of a closed convex proper cone $\gamma$ with non--empty interior. Then it is natural to endow $\BBV$ with the so-called $\gamma$-topology $\BBV_\gamma$ introduced in~\cite{KS90}*{Ch.~III~\S~5}. 
We are lead to the study of (derived) constructible $\gamma$-sheaves (that is, sheaves on $\BBV_\gamma$) and the category of such sheaves is no more equivalent to any natural category of barcodes. The aim of this paper, a kind of  continuation of~\cite{KS18}, is to find a substitute to this non existing equivalence. 

First, we replace constructible sheaves with \PL sheaves (\PL for piecewise linear) which are much easier to manipulate. A convex polyhedron is the intersection of a finite family of open or closed affine half-spaces and 
a constructible sheaf $F$ is \PL if there is a finite covering of $\BBV$ by  convex polyhedra on which it is  constant.  
It has been proved in loc.\ cit. that constructible sheaves may be approximated, for a kind of derived  bottleneck distance, by \PL sheaves and similarly for constructible $\gamma$-sheaves. A natural higher dimensional analogue to the category of $\gamma$-barcodes is given by the additive category of finite direct sums of constant sheaves on convex $\gamma$-locally closed polyhedra and the main result of this paper 
asserts that  the triangulated category of  \PL  $\gamma$-sheaves is generated by the additive category of such $\gamma$-barcodes.

This paper also contains a systematic study of $\PL$-sheaves. We show in particular that a sheaf is \PL if and only if its microsupport is a \PL Lagrangian variety or, equivalently, is contained in such a Lagrangian variety. We note that 
 the six Grothendieck operations hold for \PL sheaves and that \PL sheaves on $\BBV$ may be considered as the restriction to $\BBV$ of  \PL sheaves on its projective compactification $\BBP$. In the course of the paper, we recall several results of loc.\ cit. that we shall need and also give a new application of the stability theorem of the bottleneck distance. 


\vspace{1.ex}\noindent
{\bf Acknowledgments}\\
The authors would like to thank the anonymous referee for many valuable suggestions and corrections.
\begin{remark}\label{preremark}
Some notations and conventions in this paper  differ from those of~\cite{KS18}. 
\banum
\item
A polyhedron was called a polytope in loc.\ cit.
\item
A \PL set or a \PL sheaf in loc.\ cit.\, is called here an \LPL set or an \LPL sheaf. 
\item
The microsupport of a sheaf $F$ is denoted here by $\SSi(F)$ as in~\cite{KS90}, instead of  $\musuppb(F)$ in~\cite{KS18}.
\eanum
\end{remark}

\section{\PL geometry}\label{section:PLgeometry}

\subsection{\PL sets and \PL stratifications}
Let $\BBV$ be a real finite-dimensional vector space. 

\begin{definition}\label{def:polyhedra}
\banum
\item
A {\em convex polyhedron} $P$ in $\BBV$ 
is the intersection of a finite family of open or closed affine half-spaces. 
\item
A \PL set is a  finite union of convex polyhedra.
 \item
 A locally  \PL  set  (an \LPL set for short)  is  a  locally finite  union of convex polyhedra.
 \eanum
\end{definition}
Note that an \LPL set is subanalytic.

The next result is obvious. 
\begin{lemma}
\bnum
\item
The family of  \PL sets  in $\BBV$ is stable by finite unions and finite intersections.
\item
If $Z$ is \PL, then  its closure $\ol Z$, its interior $\Int(Z)$ and its  complementary set $\BBV \setminus Z$ are \PL.
\item Any connected component of a \PL set is \PL. 
\item \label{item:4}
Let $u\cl\BBV\to\BBW$ be a linear map. 
\bna
\item If $S\subset\BBV$ is \PL,  then $u(S)\subset\BBW$ is \PL. \label{item:a}
\item
If $Z\subset\BBW$ is  \PL, then $\opb{u}(Z)\subset\BBV$ is  \PL.
\ee
\item
The preceding results still hold  when  when replacing $\PL$ with $\LPL$, except 
\eqref{item:a} in which case one has to assume that $u$ is proper on $\ol S$.
\eanum
\end{lemma}

For a locally closed submanifold $Z\subset\BBV$, one sets for short
\eqn
&&T^*_Z\BBV\eqdot T^*_ZU \text{ where $U$ is an open subset of $\BBV$ containing $Z$ as a closed subset.}
\eneqn

\begin{definition}\label{def:PLstrat}
A  \PL-stratification of a set $S$ of $\BBV$ is a  finite  family 
$Z=\{Z_a\}_{a\in A}$ of  non-empty convex polyhedra such that
\bnum
\item  $S=\bigcup_{a\in A}Z_a$, 
\item
each $Z_a$ is a  locally closed  submanifold, 
\item 
$Z_a\cap Z_b=\varnothing$ for $a\neq b$,
\item 
$Z_a\cap \ol Z_b\neq\varnothing$ implies $Z_a\subset \ol Z_b$.
\enum
Replacing ``a finite family'' with ``a locally finite family'' we get the notion of an  \LPL-stratification.
\end{definition}

Recall  the operation  $\csum$  and the notion of a $\mu$-stratification of~\cite{KS90}*{Def.~6.2.4,~8.3.19}.
\begin{proposition}\label{pro:mustrat}
Let $Z=\{Z_a\}_{a\in A}$ be an \LPL stratification. Then 
\bnum
\item
$\{Z_a\}_{a\in A}$ is a $\mu$-stratification, that is, $Z_a\subset \ol Z_b$ implies 
$(T^*_{Z_a}\BBV\csum T^*_{Z_b}\BBV)\cap\opb{\pi}(Z_a)\subset T^*_{Z_a}\BBV$. 
\item
Set $\Lambda=\bigsqcup_{a\in A} T^*_{Z_a}\BBV$. Then $\Lambda\csum\Lambda=\Lambda$. 
\label{itm:mustrat}
\enum
\end{proposition}
\begin{proof}
We shall prove both statements together.

Assume that $Z_a\subset \ol Z_b\cap \ol Z_c$. We may assume (in a neighborhood of a point of $Z_a$) that 
$\BBV=\BBW\oplus\BBW'$ for two linear spaces $\BBW$ and $\BBW'$ and $Z_a=\BBW$. 
Then $Z_b=\BBW\times S$ and $Z_c=\BBW\times L$ where $S$ is open in some linear subspace $\BBW''$ of $\BBW'$ and similarly for   $L$. 
Then one immediately checks that $(T^*_{Z_b}\BBV\csum T^*_{Z_c}\BBV)\cap\opb{\pi}(Z_a)\subset T^*_{Z_a}\BBV$.

This proves (ii). Choosing $c=a$ we get (i).
\end{proof}

\begin{proposition}
Consider a  finite  family   $\{P_b\}_{b\in B}$ of convex polyhedra. Then there exists a \PL-stratification 
$\BBV=\bigsqcup_{a\in A} Z_a$ such that 
each  $P_b$ is a union of strata.
\end{proposition}
In the sequel, an interval of $\R$ means a convex subset of $\R$.

\begin{proof}
There exists a finite family $\{f_1,\dots,f_l\}$ of linear forms and a finite family  $\{I_c\}_{c\in C}$ such that each $I_c$ is either an open interval or a point,
$\R=\bigsqcup_{\;c\in C}I_c$ and for all $b\in B$, 
\eqn
&&P_b=\bigcap_{1\leq j\leq l}\opb{f_j}(J_{j,b}), \mbox{ where $J_{j,b}$ is a union of some $I_c$, $c\in C$}.
\eneqn
For any family $d=\{c_1,\dots,c_l\}\in C^l$, set 
\eqn
&&Z_d=\bigcap_{j=1}^l\opb{f_j}(I_{c_j}).
\eneqn
Then the family $\{Z_d\}_{d\in C^l}$ is a  \PL-stratification of $\BBV$ finer that the family $\{P_b\}_{b\in B}$. 
\end{proof}

\subsection{\PL Lagrangian subvarieties}
Recall that the notions of co-isotropic, isotropic and Lagrangian  subanalytic subvarieties are given in~\cite{KS90}*{Def.~6.5.1, 8.3.9}.

\begin{proposition}\label{prop:Flinear2}
Let $\Lambda$ be a locally closed conic  \LPL  isotropic  subset of $T^*\BBV$. Then 
 for any $p\in\Lambda_\reg$ there exists 
a linear affine subspace $L\subset\BBV$ with $\Lambda\subset  T^*_L\BBV$ in a neighborhood of $p$.
\end{proposition}
\begin{proof}
 If $\lambda$ is a linear affine isotropic subspace of $T^*\BBV$, then there exists a linear affine subspace $L$ of $\BBV$ such that $\lambda\subset T^*_L\BBV$. 
\end{proof}

\begin{lemma}\label{le:SisPL}
Let $\{L_a\}_{a\in A}$ be a  finite family of affine linear subspaces in a vector space $\BBW$. Set $X=\bigcup_{a\in A}L_a$ and let 
 $S$ be a closed subset of $X$.
 Assume that 
$S\cap X_\reg$ is open in $X_\reg$ and
$S$ is the closure of $S\cap X_\reg$. Then $S$ is \PL. 
\end{lemma}
\Proof
Indeed $S\cap X_\reg$ is a locally finite union of connected components of
$X_\reg$. 
\QED
\begin{theorem}\label{th:PLlag1}
\banum
\item
Let $\Lambda$ be a  locally closed conic  \PL\ isotropic  subset of $T^*\BBV$. Then there exists a \PL\ stratification 
 $\{P_a\}_{a\in A}$ of $\BBV$ such that $\Lambda\subset\bigsqcup_{a\in A} T^*_{P_a}\BBV$.
\item
Let $\Lambda$ be a locally  closed conic subanalytic  Lagrangian subset of $T^*\BBV$ and assume that $\Lambda$ is contained 
in a closed conic  \PL\  isotropic subset. Then $\Lambda$ is $\PL$.
 \eanum
 \end{theorem}
 \begin{proof}
 (a) Let $\{\Omega_i\}_{i\in I}$ be the family of  connected components of $\Lambda_\reg$. 
Note that the $\Omega_i$'s are \PL. Then there exists an affine linear subspace  $L_i$ such that $\Omega_i\subset T^*_{L_i}\BBV$
by Proposition~\ref{prop:Flinear2}. 
Choose a \PL stratification $\{P_a\}_{a\in A}$ finer  than 
the family $\{L_i\}_{i\in I}$.  Then $\Lambda_\reg\subset\bigsqcup T^*_{P_a}\BBV$ and  Proposition~\ref{pro:mustrat}\;\eqref{itm:mustrat} implies that this last set is closed, hence contains $\Lambda$. 
  
 \spa
 (b) follows from Lemma~\ref{le:SisPL} with $\BBW=T^*\BBV$.
  \end{proof}
 \begin{remark}
 In Lemma~\ref{le:SisPL} and Theorem~\ref{th:PLlag1}, all statements remain true when replacing everywhere ``finite'' with ``locally finite'' and ``\PL'' with ``\LPL''.
 \end{remark}

\section{\PL sheaves}\label{section:PLsheaves}

\subsection{Review on sheaves}
Let us recall some definitions extracted from~\cite{KS90} and a few notations.
\begin{itemize}
\item
Throughout this paper, $\cor$ denotes a field.
We denote by $\md[\cor]$ the abelian category of $\cor$-vector spaces. 
\item
For an abelian category $\shc$, we denote by $\Derb(\shc)$ its bounded derived category. However, 
we write $\Derb(\cor)$ instead of $\Derb(\md[\cor])$.
\item
For a vector bundle $E\to M$, we denote by $a\cl E\to E$ the antipodal map, $a(x,y)=(x,-y)$.
For a subset $Z\subset E$, we simply denote by $Z^a$ its image by the antipodal map.
In particular, 
for a cone $\gamma$ in $E$, we denote by $\gamma^a=-\gamma$ the opposite cone. For such a cone, we denote by $\gamma^\circ$ the polar cone (or dual cone) in the dual vector bundle $E^*$:
\eq\label{eq:polar}
&&\gamma^\circ=\{(x;\xi)\in E^*;\langle\xi,v\rangle\geq0\mbox{ for all }v\in\gamma_x\}.
\eneq
\item
Let $M$ be a real manifold  of dimension $\dim M$. We shall use freely the classical notions of microlocal sheaf theory, referring to~\cite{KS90}.
We denote by $\md[\cor_M]$ the abelian category of sheaves of $\cor$-modules on $M$ and by $
\Derb(\cor_M)$ its bounded derived category.
For short, an object of $\Derb(\cor_M)$ is called a ``sheaf'' on $M$.
\item
For a locally  closed subset $Z\subset M$, one denotes by $\cor_Z$ the constant sheaf with stalk $\cor$ on $Z$ extended by $0$ on $M \setminus Z$. One defines similarly the sheaf  $L_Z$ for $L\in\Derb(\cor)$.
\item
For $F\in\Derb(\cor_M)$ we denote by $\musupp(F)$ its singular support, or microsupport, a closed conic co-isotropic subset of $T^*M$.
\end{itemize}

\subsubsection*{Constructible sheaves}
We refer the reader to \cite{KS90} for terminologies not explained here. 
\begin{definition}
Let $M$ be a real analytic manifold and let $F \in \md[\cor_M]$.
One says that $F$ is weakly $\R$-constructible if there exists a subanalytic stratification $M=\bigsqcup_{a\in A} M_a$ such that for each  stratum $M_a$, the restriction $F\vert_{M_a}$ is locally constant.
If moreover, the stalk $F_x$ is of finite rank for all $x\in M$, then one says that $F$ is $\R$-constructible.
\end{definition}
\begin{notation}
(i) One denotes by $\mdrc[\cor_M]$ the abelian category of $\R$-constructible sheaves, a thick abelian subcategory of $\md[\cor_M]$.  

\spa
(ii) One denotes by $\Derb_{\Rc}(\cor_{M})$ the full triangulated subcategory of $\Derb(\cor_{M})$ consisting of sheaves with $\R$-constructible cohomology  and by  $\Derb_{\rcc}(\cor_{M})$ the full triangulated subcategory of $\Derb_\Rc(\cor_{M})$ consisting of sheaves with compact support.
\end{notation}
Recall that the natural functor $\Derb(\mdrc[\cor_M])\to\Derb_\Rc(\cor_{M})$ is an equivalence of categories.

\subsection{Microlocal characterization of \PL sheaves}
Recall Remark~\ref{preremark}.
\begin{definition}\label{def:Flinear}
One says that $F\in \Derb(\cor_\BBV)$ is   \PL if there exists a  finite
 family $\{P_a\}_{a\in A}$ of
convex polyhedra such that  $\BBV=\bigcup_{a\in A}P_a$ and 
$F\vert_{P_a}$ is constant of finite rank for any $a\in A$. 

Replacing the finite family $\{P_a\}_{a\in A}$ with a locally finite family, we get the notion of an \LPL sheaf.
\end{definition}
 By this definition 
\eq\label{eq:PLH0PL}
&&\parbox{70ex}{
$F$ is \PL if and only if  $H^j(F)$ is \PL for all $j\in\Z$. 
}\eneq 

One sets 
\eq\label{not:2}
&&\left\{\begin{array}{l}
\Derb_{\rcl}(\cor_{\BBV})\eqdot\{F\in\Derb(\cor_{\BBV}) ;\mbox{ $F$ is \PL}\},\\[1ex]
\mdrcl[\cor_{\BBV}]\eqdot\md[\cor_{\BBV}]\cap \Derb_{\rcl}(\cor_\BBV).
\end{array}\right.
\eneq
Of course, $\Derb_{\rcl}(\cor_{\BBV})$ is a subcategory of $\Derb_{\Rc}(\cor_{\BBV})$ and 
$\mdrcl[\cor_{\BBV}]$ is a subcategory of $\mdrc[\cor_{\BBV}]$.

\begin{proposition}
The natural functor 
$\Derb(\mdrcl[\cor_{\BBV}])\to \Derb_{\rcl}(\cor_{\BBV})$ is an equivalence.
\end{proposition}
 \begin{proof}
There exists a triangulation $\mathbb S=(S,\Delta)$
and a homeomorphism $f\cl |\mathbb S|\to \BBV$ such that
its restriction to $|\sigma|$ is linear for any $\sigma\in\Delta$.
Then the result follow from \cite[Th.~8.1.10]{KS90}.
 \end{proof}

\begin{theorem}\label{th:PL}
Let $F\in\Derb_{\Rc}(\cor_{\BBV})$. Then the conditions below are equivalent.
\banum
\item
$F\in\Derb_{\rcl}(\cor_{\BBV})$,
\item
$\musupp(F)$ is a  closed conic  \PL\ Lagrangian subset of $T^*\BBV$,
\item
$\musupp(F)$ is contained in a closed conic  \PL\  isotropic subset of $T^*\BBV$.
\eanum
\end{theorem}
 \begin{proof}
 (a)$\Rightarrow$(c)
 Consider a  finite covering $\{P_b\}_{b\in B}$  by convex polyhedra such that $F\vert_{P_b}$ is constant and choose a finer \PL stratification $\BBV=\bigsqcup_{a\in A}Z_a$. This is a $\mu$-stratification and this implies $\musupp(F)\subset\bigsqcup_{a\in A}T^*_{Z_a}\BBV$ by~\cite{KS90}*{Prop.~8.4.1}. 

 \spa
 (b)$\Rightarrow$(a) By Theorem ~\ref{th:PLlag1}~(a), there exists a \PL stratification  $\BBV=\bigsqcup_{a\in A}Z_a$ such that 
 $\musupp(F)\subset \bigsqcup_{a\in A}T^*_{Z_a}\BBV$. Then $F\vert_{Z_a}$ is locally constant for each $a\in A$ by~\cite{KS90}*{Prop.~8.4.1}.

 \spa 
 (b)$\Leftrightarrow$(c) in view of Theorem~\ref{th:PLlag1}~(b).
 \end{proof}
  
The next result immediately follows from Definition~\ref{def:Flinear}. It can also easily be deduced from ~\cite{KS90}, Theorem~\ref{th:PLlag1} and Theorem~\ref{th:PL}.
\begin{corollary}\label{cor:PL}
\bnum
\item
The category $\Derb_{\rcl}(\cor_{\BBV})$ is a full triangulated subcategory of the category $\Derb(\cor_{\BBV})$ and the category $\mdrcl[\cor_{\BBV}]$ is a full thick abelian subcategory of the category $\md[\cor_{\BBV}]$.
\item 
If $F_1$ and $F_2$ are \PL, then so are $F_1\tens F_2$ and $\rhom(F_1,F_2)$.
\item
 Let $f\cl \BBV\to\BBW$ be a linear map.
\bna
\item If $G$ is a PL\ sheaf on $\BBW$, then $\opb{f}G$ and $\epb{f}G$ are 
 \PL\ sheaves  on $\BBV$.
\item If $F$ is a \PL\ sheaf on $\BBV$ then $\roim{f}F$  and $\reim{f}F$ are  \PL\ sheaves on $\BBW$.
\ee
\enum
\end{corollary}
Recall that the convolution functor  $\star\cl \Derb(\cor_\BBV)\times \Derb(\cor_\BBV)\to\Derb(\cor_\BBV)$ is defined  by the formula:
\eqn
F\star G&\eqdot&\reim{s}(F\etens G),
\eneqn
where $s\cl\BBV\times\BBV\to\BBV$ is the map $(x,y)\mapsto x+y$. 

\begin{corollary}\label{cor:PL2}
The convolution functor induces a functor 
 $\star\cl \Derb_{\rcl}(\cor_\BBV)\times \Derb_{\rcl}(\cor_\BBV)\to\Derb_{\rcl}(\cor_\BBV)$.
\end{corollary}

\begin{proposition}
Let $Z$ be a locally closed subset of $\BBV$. Then $Z$ is \PL if and only if $\musupp(\cor_Z)$ is \PL.
\end{proposition}
\begin{proof}
(i) Assume that $Z$ is \PL. Then the sheaf $\cor_Z$ is \PL and its microsupport is \PL by Theorem~\ref{th:PL}.

\spa
(ii) Conversely, assume that  $\musupp(\cor_Z)$ is \PL. Set $\partial Z=\ol Z\setminus Z$. Since $Z$ is locally closed, $\partial Z$ is closed.

\spa
(ii)--(a) First, notice that $\ol Z=\pi(\musupp(\cor_Z))$ is \PL. 

\spa
(ii)--(b) Now consider the exact sequence of sheaves $0\to\cor_Z\to\cor_{\ol Z}\to\cor_{\partial Z}\to 0$. Since  $\cor_Z$ and $\cor_{\ol Z}$ are \PL sheaves, the sheaf $\cor_{\partial Z}$ is \PL. Therefore, $\partial Z$ is \PL and it follows that $Z$ is \PL.
\end{proof}

\subsection{\PL sheaves on the projective space}

Let $\BBV\into\BBP$ be the projective compactification of $\BBV$. Hence, setting $\BBW=\BBV\times\R$, 
\eqn
&&\BBP\simeq(\BBW\setminus\{0\})/\R^\times,
\eneqn
where $\R^\times$ is the multiplicative group $\R\setminus\{0\}$. 
Denote by $\pi\cl \BBW\setminus\{0\}\to\BBP$ the projection and by
 $\iota\cl \BBW\setminus\{0\}\into\BBW$ the embedding. 

We shall say that  a subset $A$ of $\BBP$ is $\PL$ if $\iota(\opb{\pi}(A))$ is \PL in  $\BBW$.

Similarly, one defines the category of \PL sheaves $\Derb_{\rcl}(\cor_\BBP)$ as the full subcategory of 
$\Derb_\Rc(\cor_\BBP)$ consisting of objects $F$ such that $\eim{\iota}\opb{\pi}F\in \Derb_{\rcl}(\cor_{\BBW})$. 

Denote by $j\cl\BBV\into\BBP$ the open embedding and by $H_\infty=\BBP\setminus j(\BBV)$ the hyperplane at infinity.
The next result is easy and its proof is left to the reader.
\begin{proposition}
The functor $\eim{j}\cl \Derb_{\rcl}(\cor_\BBV)\to \Derb_{\rcl}(\cor_\BBP)$ is well-defined, is fully faithful and its essential image consists of objects $F$ such that $F\vert_{H_\infty}\simeq0$.
\end{proposition}

\subsection{Distance, approximation and stability}
The bottleneck distance for persistent modules  is a classical subject that we shall not review here,
referring  to~\cite{CEH07} and~\cites{CCGGO09, CSGO16, Cu13, EH08, Gri08, Le15}. Here, we use a convolution distance for sheaves similar to that of 
\cite{KS18} and slightly different from the classical ones 
since it is defined in the derived setting. Note that this ``derived'' distance has recently been studied with great details 
in~\cite{BG18} in case of dimension one.

Assume that  the vector space $\BBV$  is endowed 
with a norm $\vvert\cdot\vvert$ (see Remark~\ref{Rem:norm} below). 
We define a family of sheaves $\{K_a\}_{a\in\R}$  as follows: 
\eq\label{eq:Ka}
K_a\simeq 
\begin{cases}\cor_{\{\Vert x\Vert\leq a\}}&\text{for $a\geq0$,}\\
\cor_{\{\Vert -x\Vert<-a\}}[\dim\BBV]&\text{for $a<0$}.
\end{cases}
\eneq
There are natural morphisms and isomorphisms
\eq\label{eq:KaKb}
&&
\ba{l}
\chi_{a,b}\cl K_b\to K_a \mbox{ for }a\leq b\in\R,\\[1ex]
K_a\star K_b\simeq K_{a+b}\mbox{ for }a,b\in\R
\ea
\eneq
 such that $\chi_{a,b}\circ\chi_{b,c}=\chi_{a,c}$ for $a\le b\le c$. 
\begin{remark}\label{Rem:norm}
In~\cite{KS18} we have used the Euclidean norm, but the argument works for any norm, since~\eqref{eq:KaKb} remains true.
Here a norm $\Vert\cdot\Vert\cl \BBV\to \R_{\ge0}$
satisfies
(1) $\Vert x+y\Vert\le\Vert x\Vert+\Vert y\Vert$,
(2) $\Vert ax\Vert=a\Vert x\Vert$ for $a\ge0$, and
(3) $\Vert x\Vert=0$ implies $x=0$.
Hence norms correspond bijectively with
open relatively compact convex neighborhood of $0$.
Note that we do not ask $\Vert x\Vert=\Vert -x\Vert$ and that is why we define $K_a$ for $a<0$ as above, in order that 
$K_a\star K_{-a}\simeq K_0$.
\end{remark}

\begin{definition}{(\cite{KS18}*{Def.~3.2})}\label{def:convdist}
Let $F, G \in \Derb(\cor_{\BBV})$ and let $a\geq0$.
One says that $F$ and $G$ are {\em $a$-isomorphic} if there are morphisms 
$f\cl K_{a}\star F\to G$ and $g\cl K_a\star G\to F$ 
which satisfies the following compatibility conditions:
the composition 
$K_{2a}\star F\To[K_a\star f] K_a\star G\To[g] F$
 coincides with  the natural morphism $\chi_{0,2a}\star F\cl
K_{2a}\star F\to F$ and the composition 
$K_{2a}\star G\To[K_a\star g] K_a\star F\To[f] G$
coincides with  the natural morphism $\chi_{0,2a}\star G\cl K_{2a}\star G\to G$. 
 \end{definition}

One sets 
\eqn&&\dist(F,G)=\inf\Bigl(\{+\infty\}\cup\{a\in\R_{\geq0}\,;\,
\text{$F$ and $G$ are $a$-isomorphic}\}\Bigr)
\eneqn
and calls $\dist(\scbul,\scbul)$ the \define{convolution distance}.

In loc.\ cit.\ we have proved that any object of $\Derb_{\Rc}(\cor_\BBV)$ can be approximated with \LPL-sheaves. A similar result holds for constructible sheaves on $\BBV_\infty$.

\begin{notation}\label{not:Vinfty}
Let us denote by $\Derb_{\Rc}(\cor_{\BBV_\infty})$ the full subcategory of $\Derb_{\Rc}(\cor_\BBV)$ consisting of 
objects $F$ such that there exists $G\in\Derb_{\Rc}(\cor_\BBP)$ with $F\simeq G\vert_\BBV$. 
\end{notation}
\begin{theorem}[{The approximation theorem}]\label{th:approx}
Let  $F\in\Derb_\Rc(\cor_{\BBV_\infty})$. For each $\epsilon>0$ there exists 
$G\in \Derb_{\rcl}(\cor_\BBV)$ such that $\dist(F,G)\leq\epsilon$
and $\supp(G)\subset \supp(F)+B_\eps$, where  $B_\eps=\{x;\vvert x\vvert\leq\eps\}$. 
\end{theorem}
The proof is the same as in \cite{KS18} after noticing that one can choose the simplicial complex $\bS=(S,\Delta)$ (with the notations of loc.\ cit.) to be finite, thanks to the fact that $F\in\Derb_\Rc(\cor_{\BBV_\infty})$. 

Recall the stability theorem of~\cite{KS18}*{Th.~3.7}, a derived version of a classical theorem (see~\cite{CEH07}).

For a set $X$ and a map $f\cl X\to\BBV$, one sets
\eqn
&&\vvert f\vvert=\sup_{x\in X}\vvert f(x)\vvert.
\eneqn
 \begin{theorem}[{The stability theorem}]\label{th:stab}
 Let $X$ be a locally compact space and  let
$f_1,f_2\cl X\to\BBV$ be two continuous maps.
Then,  for any $F\in\Derb(\cor_X)$, we have
$$\dist(\roim{f_1}F,\roim{f_2}F)\leq \vvert f_1-f_2\vvert\qtq \dist(\reim{f_1}F,\reim{f_2}F)\leq\vvert f_1-f_2\vvert.$$
\end{theorem}

\subsection{Generators for \PL sheaves}
Consider a  triangulated category $\shd$ and a family of objects $\shg$. Consider the full subcategory $\sht$ of $\shd$ defined as follows. An object $F\in\shd$ belongs to $\sht$ if there exists  a finite sequence $F_0,\dots,F_N$ in $\shd$  with $F_0=0$, $F_N=F$ and distinguished triangles (d.\ t.\ for short) $F_{k}\to F_{k+1}\to G_k[m_k]\to[+1]$, $0\leq k<N$  with  $m_k\in\Z$ and $G_k\in\shg$. 

The next result is well-known from the specialists but we recall its proof for the reader's convenience. 
\begin{lemma}
The subcategory $\sht$ of $\shd$\ is triangulated. It is the smallest triangulated subcategory of $\shd$ which contains $\shg$.
\end{lemma}
\begin{proof}
We may assume from the beginning that $\shg$ is an additive category stable by isomorphisms and by shifts. Define the additive subcategory $\sht_n$ by induction as follows:
\eqn
&&\parbox{70ex}{$\sht_0=\{0\}$, $X\in\sht_n$ if and only if there is a d.\ t.\ $X_{n-1}\to X\to X_1\to[+1]$ with 
$X_{n-1}\in\sht_{n-1}$ and $X_1\in\shg$.
}\eneqn
Clearly, $\sht_1=\shg$. 

Consider a d.\ t.\ $X_n\to X\to[\alpha]X_p\to[+1]$ with $X_i\in\sht_i$ as above. By the definition of $\sht_p$, there exists a d.\ t.\ 
$X_{p-1}\to X_p\to[\beta]X_1\to[+1]$ with $X_i\in\sht_i$, $i=p-1,1$. Consider a d.\ t.\ $Y\to X\to[\beta\circ\alpha]X_1\to[+1]$. By the octahedral axiom we get a d.\ t.\ $X_n\to Y\to X_{p-1}\to[+1]$. By the induction hypothesis, $Y\in\sht_{n+p-1}$. Hence $X\in\sht_{n+p}$. 
\end{proof}

In this paper, we shall say that  $\shg$ generates $\shd$ if $\sht=\shd$. 

\begin{theorem}\label{th:genertri}
The triangulated category  $\Derb_{\rcl}(\cor_{\BBV})$ is generated by the family $\{\cor_P\}$ where $P$ ranges over the family of locally closed convex polyhedra.
\end{theorem}
\begin{proof}
(i) We denote by $\shg$ the family of sheaves isomorphic to some  $\cor_P$, $P$  a locally closed convex polyhedron, and denote by $\sht$ the triangulated subcategory of $\Derb_{\rcl}(\cor_{\BBV})$ generated by $\shg$. 

\spa
(ii) We argue by induction on $\dim \BBV$. The case where $\dim\BBV$ is $0$  is clear.

\spa
(iii) Let $F\in\Derb_{\rcl}(\cor_{\BBV})$. By truncation, we may reduce to the case where $F$ is concentrated in degree $0$. 

\spa
(iv) There exists a finite family $\{H_a\}_{a\in A}$ 
 of affine hyperplanes 
such that, setting
 $U= \BBV\setminus\bigcup_aH_a$,  the restriction of $F$ to $U$ is locally constant. Let $U=\bigsqcup_iU_i$ be the decomposition of $U$ into connected component. Each $U_i$ is an open convex polyhedron. Set  $Z=\bigcup_a H_a$ and consider the exact sequence
$0\to F_U\to F\to F_Z\to 0$. The sheaf $F_U$ is a finite direct sum of sheaves of the type $\cor_{U_i}$. Hence $F_U\in\sht$  and it remains to show that $F_Z$ belongs to $\sht$. 

\spa
(v) We argue by induction on $\# A$. If $\# A=1$, then the result follows from the induction hypothesis on the dimension of $\BBV$ since we may identify $F_{H_a}$ with a sheaf on the affine space $H_a$. 
Let $a\in A$ and define $G$ by the exact sequence $0\to G\to F_Z\to F_{H_a}\to 0$. 
By the induction hypothesis $G$ and $F_{H_a}$ belong to $\sht$ and the result follows.
\end{proof}

\section{ \PL $\gamma$-sheaves}\label{section:PLgamma}

\subsection{Review on $\gamma$-sheaves}\label{subsection:gammashv}

In this subsection we shall review some definitions and results extracted from~\cite{KS90, KS18}. 
The so-called $\gamma$-topology has been studied with some details in~\cite{KS90}*{\S\,3.4}.

Let $\BBV$ be a  finite-dimensional  real vector space.
Recall that we denote by $a\cl \BBV\to\BBV$ the antipodal map $x\mapsto -x$ and by $s\cl\BBV\times\BBV\to\BBV$ the map $(x,y)\mapsto x+y$. 
Hence,  for two subsets $A,B$ of $\BBV$, one has $A+B=s(A\times B)$. 

A subset $A$ of $\BBV$  is called a cone if $0\in A$ and $\R_{>0}A\subset A$. A convex cone  $A$ is proper
if $A\cap A^a=\{0\}$.

Throughout  the rest of the paper, we consider a cone $\gamma\subset\BBV$ 
 and we assume:
 \eq\label{hyp0}
&&\parbox{75ex}{\em{
$\gamma$ is  closed proper convex with non-empty interior. }
}\eneq
In \S~\ref{subsection:PLgammashv} we shall make the extra assumption that $\gamma$ is {\em polyhedral}, meaning that it is a finite intersection of closed  half-spaces.

We say that a subset $A$ of $\BBV$ is {\em $\gamma$-invariant} if $A+\gamma=A$.
Note that a subset $A$ is $\gamma$-invariant if and only if
$\BBV\setminus A$ is $\gamma^a$-invariant.

The family of $\gamma$-invariant open subsets of $\BBV$ defines a topology, which is called  the \define{$\gamma$-topology}
 on $\BBV$.
One denotes by $\BBV_\gamma$ the space $\BBV$ endowed with the $\gamma$-topology and one denotes by
\begin{equation}
\phig \cl\BBV \to \BBV_\gamma
\end{equation}
the continuous map associated with the identity.
Note that the closed sets for this topology are the 
$\gamma^a$-invariant closed subsets of $\BBV$. 
 
\begin{definition}\label{def:loclo}
Let $A$ be a subset of $\BBV$. 
\banum
\item
One says that $A$ is {\em$\gamma$-open} \lp resp.\ {\em $\gamma$-closed}\rp\, if $A$ is open \lp resp.\  closed\rp\,  for the $\gamma$-topology. 
\item
One says that $A$ is {\em $\gamma$-locally closed} if $A$ is the intersection of  a $\gamma$-open subset and a $\gamma$-closed  subset.

\item
One says that $A$ is {\em $\gamma$-flat} if $A=(A+\gamma)\cap(A+\gamma^a)$.
\item 
One says that a closed set $A$ is {\em $\gamma$-proper} if the addition  map $s$ is proper on $A\times\gamma^a$.
\eanum
\end{definition}
Remark that a closed subset $A$ is $\gamma$-proper if and only if
$A\cap(x+\gamma)$ is compact for any $x\in\BBV$.

\begin{proposition}[\cite{KS18}*{Prop.~4.4}]
The set of $\gamma$-flat open subsets $\Omega$ of $\BBV$ and
the set of $\gamma$-locally closed subsets $Z$ of $\BBV$ 
correspond bijectively by
$$\xymatrix@R=0ex@C=5ex{
\mbox{$\rule{4ex}{0ex}\Omega$}\ar@{|->}[r]&(\Omega+\gamma)\cap 
\ol{\;\Omega+\gamma^a\;}\\
\Int(Z)&\mbox{$Z.\rule{14ex}{0ex}$}\ar@{|->}[l]}
$$ 
In particular, $\gamma$-locally closed subsets are  $\gamma$-flat.
\end{proposition}

We shall use the notations:
\eq\label{not:0}
&&\left\{\begin{array}{l}
\Derb_{\gammac}(\cor_{\BBV})\eqdot\{F\in\Derb(\cor_{\BBV}) ;\musupp(F)\subset \BBV\times\gammac\},\\[1ex]
\Derb_{\rcg}(\cor_{\BBV})\eqdot\Derb_{\Rc}(\cor_{\BBV})\cap\Derb_{\gammac}(\cor_{\BBV}),\\[1ex]
\mdg[\cor_{\BBV}]\eqdot\md[\cor_{\BBV}]\cap \Derb_{\gammac}(\cor_\BBV),\\[1ex]
\mdrcg[\cor_{\BBV}]\eqdot\mdrc[\cor_\BBV]\cap\mdg[\cor_{\BBV}].
\end{array}\right.
\eneq
We call an object of $\Derb_{\gammac}(\cor_{\BBV})$ \define{a $\gamma$-sheaf}. 

It follows from~\cite{KS90}*{Prop.~5.4.14} that for $F,G\in\Derb_{\gammac}(\cor_{\BBV})$ and $H\in\Derb_{\gamma^\circ}(\cor_{\BBV})$, the sheaves $F\tens G$ and $\rhom(H,F)$ belong to $\Derb_{\gammac}(\cor_{\BBV})$.

The next result is implicitly proved in~\cite{KS90} and explicitly in~\cite{KS18}. (In this statement, the hypothesis that  
$\Int(\gamma)$ is non empty is not necessary.)

\begin{theorem}\label{th:eqvderbg}
Let $\gamma$ be a closed convex proper cone in $\BBV$.
The  functor $\roim{\phig}\cl \Derb_{\gammac}(\cor_\BBV)\to\Derb(\cor_{\BBV_\gamma})$ is an equivalence of triangulated categories with quasi-inverse $\opb{\phig}$. Moreover, this equivalence preserves the natural $t$-structures of both categories. In particular, for $F\in \Derb(\cor_\BBV)$, the condition
$F\in \Derb_{\gammac}(\cor_\BBV)$ is equivalent to the condition:
$\musupp(H^j(F))\subset \BBV\times\gamma^{\circ a}$  for any $j\in\Z$.
\end{theorem}
Thanks to this theorem, the reader may ignore microlocal sheaf theory, at least in a first reading.

\begin{corollary}[\cite{KS18}*{Cor.~2.8}]\label{cor:ssA}
Let $A$ be a $\gamma$-locally closed subset of \/ $\BBV$.
Then $\musupp(\cor_A)\subset \BBV\times\gammac$. 
\end{corollary}

\begin{proposition} \label{pro:lipsch}
Assume~\eqref{hyp0}. Let $U=U+\gamma$ be a $\gamma$-open set and let $x_0\in\partial U$. Then there exist
a linear coordinate system $(x_1,\dots,x_n)$ on $\BBV$, 
an open neighborhood $V$ of $x_0$, an open subset  $W$ of $\BBV$ and a bi-Lipschitz isomorphism $\phi\cl V\isoto W$  such that 
$\phi(V\cap U)=W\cap\{x\in\BBV;x_n>0\}$.
\end{proposition}
\begin{proof}
The proofs of~\cite{GS16}*{Lem.~2.36, 2.37} (which were formulated for subanalytic open subsets) extend immediately to our situation. 
\end{proof}

Recall the duality functor
\eqn
&&\RD'_M(\scbul)=\rhom(\scbul,\cor_M).
\eneqn
\begin{corollary}
Let $Z$ be a $\gamma$-locally closed subset of \/ $\BBV$. 
Then,
$\RD'_M(\cor_Z)$ is concentrated in degree $0$. Moreover, 
$\RD'_M(\cor_Z)\simeq \cor_S$ with 
 $\Omega=\Int(Z)$ and $S=\ol{\Omega+\gamma}\cap (\Omega+\gamma^a)$. 
\end{corollary}
\begin{proof}
It follows from Proposition~\ref{pro:lipsch} that $\RD'_M(\cor_{\Omega+\gamma})\simeq \cor_{\ol{\Omega+\gamma}}$
and $\RD'_M(\cor_{\ol{\Omega+\gamma^a}})\simeq\cor_{\Omega+\gamma^a}$. 

Set $A=\Omega+\gamma$ and $B= \ol{\Omega+\gamma^a}$. 
 Then $\cor_A$ and $\cor_B$ are cohomologically constructible. 
By Corollary~\ref{cor:ssA}, $\musupp(\cor_A)\cap\musupp(\cor_B)\subset T^*_\BBV\BBV$. Then $\RD'_M(\cor_A\tens\cor_B)\simeq \RD_M'(\cor_A)\tens \RD_M'(\cor_B)$ 
by~\cite{KS90}*{Cor.~6.4.3}.
\end{proof}

\subsection{An application of the stability theorem}

Here we give an application of the stability theorem (Theorem~\ref{th:stab}) which did not appear in ~\cite{KS18}. Of course, this application is well-known 
for the classical (non-derived) distance. 

Let $(M,d)$ be a subanalytic metric space and let $K_1$ and $K_2$ be two subanalytic compact subsets. 
Define for $i=1,2$
\eqn
&&f_i(\cdot)=d(\cdot,K_i),\\
&& \Gamma_i= \{(x,t)\in M\times\R;d(x,K_i)= t\},\quad G_i=\cor_{\Gamma_i}, \\
&& Z_i=\{(x,t)\in M\times\R;d(x,K_i)\leq t\},\quad F_i=\cor_{Z_i}.
\eneqn
Hence, $\Gamma_i$ is the graph of $f_i$ and $Z_i$ is the epigraph of $f_i$.

\begin{lemma}\label{le:dist1}
One has $d(K_1,K_2)=\vvert f_1-f_2\vvert$. 
\end{lemma}
\begin{lemma}\label{le:dist2}
One has $\dist(\roim{f_1}\cor_M, \roim{f_2}\cor_M)\leq\vvert f_1-f_2\vvert$. 
\end{lemma}
\begin{proof}
This is a particular case of  Theorem~\ref{th:stab}.
\end{proof}

\begin{lemma}\label{le:dist3}
Denote by $\gamma$ the cone $\{t\leq0\}$ in $\R$ and still denote by $\phig$ the map $M\times\R\to M\times\R_\gamma$. Then
$F_i\simeq\opb{\phig}\roim{\phig}G_i$. 
\end{lemma}
Let $p$ denote the projection $M\times\R\to\R$.
\begin{lemma}\label{le:dist4}
One has $\dist(\roim{p}\opb{\phig}\roim{\phig}G_1, \roim{p}\opb{\phig}\roim{\phig}G_2)\leq \dist(\roim{p}G_1,\roim{p}G_2)$. 
\end{lemma}
\begin{proof}
The two functors $\roim{p}$ and $\opb{\phig}\roim{\phig}$ commute (with obvious notations). Then the result follows 
from~\cite{KS18}*{Prop.~3.6}. 
\end{proof}
Applying these lemmas, one gets 
 a derived version of a result of~\cite{CEH07}.

\begin{theorem}
One has $\dist(\roim{p}F_1,\roim{p}F_2)\leq d(K_1,K_2)$. 
\end{theorem}
\begin{proof}
We have 
\eqn
&&\roim{p}G_i\simeq\roim{f_i}\cor_M.
\eneqn
Applying Lemma~\ref{le:dist2}, we get
\eqn
&&\dist(\roim{p}G_1,\roim{p}G_2)\leq \vvert f_1-f_2\vvert. 
\eneqn
To conclude, apply Lemma~\ref{le:dist3},~\ref{le:dist4} and~\ref{le:dist1}.

\end{proof}

\subsection{Review on \PL $\gamma$-sheaves}\label{subsection:PLgammashv}

From now on, we shall assume that the cone $\gamma$ satisfies:
\eq\label{hyp1}
&&\parbox{75ex}{\em{
$\gamma$ is  a closed proper convex  polyhedral cone with non-empty interior. }
}\eneq

\begin{definition}\label{def:bar}
Assume~\eqref{hyp1}.
\banum
\item
A {\em \PL $\gamma$-barcode} $(A, Z)$ in $\BBV$
is the data of  a {\em finite} set of indices  $A$ and a  family 
$Z=\{Z_a\}_{a\in A}$ of    non-empty,  $\gamma$-locally closed, 
convex polyhedra of $\BBV$. 
\item
A disjoint $\gamma$-barcode is
a $\gamma$-barcode $(A, Z)$ such that 
$Z_a\cap Z_b=\varnothing$ for $a\neq b$. 
\item
The support of a  $\gamma$-barcode  $(A, Z)$, denoted  by  $\supp(A,Z)$, is the set 
$\bigcup_{a\in A}\ol{Z_a}$. 
\eanum
\end{definition}
\begin{remark}
In~\cite{KS18}, we defined a \PL $\gamma$-stratification of a closed set $S$ as a barcode $(A,Z)$ such that $\supp(A,Z)=S$ and $Z_a\cap Z_b=\varnothing$ for $a\neq b$. However, since a  \PL $\gamma$-stratification is not a \PL stratification (see Definition~\ref{def:PLstrat}), we prefer here to avoid this terminology and use the notion of a  disjoint $\gamma$-barcode.
\end{remark}

We shall use the notations:
\eq\label{not:3}
&&\left\{\begin{array}{l}
\Derb_{\rclg}(\cor_{\BBV})\eqdot\Derb_{\rcl}(\cor_{\BBV})\cap\Derb_{\gammac}(\cor_{\BBV}),\\[1ex]
\mdrclg[\cor_{\BBV}]\eqdot\md[\cor_\BBV]\cap\Derb_{\rclg}(\cor_{\BBV}).
\end{array}\right.
\eneq
Note that, in view of~\eqref{eq:PLH0PL} and Theorem~\ref{th:eqvderbg}:
\eq\label{eq:PLgammaH0}
&&F\in \Derb_{\rclg}(\cor_{\BBV})\Leftrightarrow H^j(F)\in \mdrclg[\cor_{\BBV}]\mbox{ for all }j\in\Z.
\eneq

\begin{proposition}\label{pro:phigPL}
Assume~\eqref{hyp1}. If  $F\in \Derb_{\rcl}(\cor_{\BBV})$
then $\opb{\phig}\roim{\phig}F\in \Derb_{\rclg}(\cor_{\BBV})$.
\end{proposition}
\begin{proof}
By Theorem~\ref{th:eqvderbg}, it remains to prove that $\opb{\phig}\roim{\phig}F$ is \PL. 
Denote by $Z(\gamma)$ the set $\{(x,y)\in\BBV\times\BBV;y-x\in\gamma\}$ and denote by $q_1$ and $q_2$ the first and second projections defined on $\BBV\times\BBV$. Then (see~\cite{KS90}*{Prop.~3.5.4}):
\eqn
&&\opb{\phig}\roim{\phig}F\simeq \roim{q_1}(\cor_{Z(\gamma)}\tens \opb{q_2}F)
\eneqn
and the result follows from Corollary~\ref{cor:PL}.
\end{proof}

Recall Notation~\ref{not:Vinfty}.

\begin{corollary}\label{cor:phigPL}
Let  $F\in\Derb_{\rcg}(\cor_{\BBV_\infty})$. 
For each $\epsilon>0$ there exists 
$G\in \Derb_{\rclg}(\cor_\BBV)$ such that $\dist(F,G)\leq\epsilon$. 
\end{corollary}
\begin{proof}
We shall apply Theorem~\ref{th:approx}.
There exists $G' \in \Derb_{\rcl}(\cor_\BBV)$ such that
$\dist(F,G')\leq\epsilon$. 
Then, $G\eqdot \opb{\phig}\roim{\phig}G'\in \Derb_{\rclg}(\cor_\BBV)$ 
by Proposition~\ref{pro:phigPL}. Moreover, by~\cite{KS18}*{Prop.~3.6}:
$$\dist\bl\opb{\phig}\roim{\phig}F,\opb{\phig}\roim{\phig} G'\br
\le\dist(F,G')\le\epsilon.$$
Since $F$ is a $\gamma$-sheaf, one has
$\opb{\phig}\roim{\phig}F\simeq F$.
\end{proof}

The three theorems below are the main results of~\cite{KS18}. We recall them for the reader's convenience. 
 
\begin{theorem}[\cite{KS18}*{Th.~4.10}]\label{th:germofF}
Assume~\eqref{hyp1} and let $F\in \Derb_{\rcg}(\cor_{\BBV})$. Then for each $x\in\BBV$, there exists an open neighborhood $U$ of $x$ such that $F\vert_{(x+\gamma^a)\cap U}$ is constant. 
\end{theorem}

\begin{theorem}[\cite{KS18}*{Th.~4.14}] \label{th:locctonflat}
Assume~\eqref{hyp1} and let $F\in \Derb_{\rcg}(\cor_{\BBV})$. 
 Let $\Omega$ be a $\gamma$-flat open set and let 
$Z=(\Omega+\gamma)\cap\ol{\;\Omega+\gamma^a\;}$,  a $\gamma$-locally closed subset. 
Assume that  $F\vert_\Omega$ is locally constant. Then $F\vert_Z$ is locally constant. 
\end{theorem}

\begin{theorem}[\cite{KS18}*{Lem.~4.16, Th.~4.17}]\label{th:linearcase}
Assume~\eqref{hyp1} and let $F\in\Derb_{\rclg}(\cor_\BBV)$. Then 
there exists a disjoint $\gamma$-barcode $(A,Z)$ with $\supp(A,Z)=\supp(F)$ and 
such that $F\vert_{Z_a}$ is constant for each $a\in A$. Moreover, $F_x\simeq 0$ for $x\notin\bigsqcup_{a\in A}Z_a$. 
\end{theorem}
(In fact, Theorem~\ref{th:linearcase} was proved for \LPL sheaves but the proof can easily be adapted to \PL sheaves.)

\subsection{Generators for \PL $\gamma$-sheaves}

In~\cite{KS18} we have constructed a category $\Barc$ whose objects are the $\gamma$-barcodes and a  fully faithful  functor 
\eq\label{eq:fctpsi}
&&\Psi\cl \Barc\to\mdrclg[\cor_\BBV],\quad
Z=\{Z_a\}_{a\in A}\mapsto \soplus_{a\in A}\cor_{Z_a}.
\eneq
However, as shown in~\cite{KS18}*{Ex.~2.14, 2.15}, the functor $\Psi$ is not essentially surjective as soon as $\dim\BBV>1$. 

\begin{definition}\label{def:barcodesheaf}
An object of $\mdrclg[\cor_\BBV]$ is a barcode $\gamma$-sheaf if it is in the essential image of $\Psi$.
\end{definition}

In~\cite{KS18} we made the following conjecture.
\begin{conjecture}
Let $F\in\Derb_{\rclg}(\cor_\BBV)$ and assume that $F$ has compact support. Then there exists a bounded complex 
$F^\scbul\in \RC^\rb(\mdrclg[\cor_\Vg])$ whose image in  $\Derb_{\rclg}(\cor_\BBV)$  is isomorphic to $F$ and such that each component $F^j$ 
of $F^\scbul$ is a barcode $\gamma$-sheaf with compact support. 
\end{conjecture}
As usual, for an additive category $\shc$,   $\RC^\rb(\shc)$ denotes the category of bounded complexes of objects of $\shc$.

In this subsection, we shall prove a weaker form of this conjecture, namely:

\begin{theorem}\label{th:genergammatri}
The triangulated category  $\Derb_{\rclg}(\cor_{\BBV})$ is generated by the family $\{\cor_P\}_P$ where $P$ ranges over the family of $\gamma$-locally closed convex polyhedra.
\end{theorem}
In particular, the category $\Derb_{\rclg}(\cor_{\BBV})$ is generated by the barcodes  $\gamma$-sheaves. 

\begin{proof}
Let $F\in\mdrclg[\cor_\BBV]$. There exists $\{\xi_k\}_{1\leq k\leq l}$ in $\gamma^{0a}$  and 
$\{c_j\}_{0\leq j\leq N}$ in $\R$ with  $-\infty=c_0<c_1<\dots<c_{N-1}<c_N=+\infty$ such that, setting
\eqn
&&H_{k,j}=\{x\in\BBV;\langle x,\xi_k\rangle=c_j\}, \quad U\eqdot \BBV\setminus\bigsqcup_{k,j}H_{k,j},
\eneqn
the sheaf $F\vert_U$ is locally constant. 

For $n=(n_1,\dots,n_l)$ with $0\leq n_k<N$, define
\eqn
&&Z_n=\bigcap_{k=1}^l\{x;c_{n_k}\leq \langle x,\xi_k\rangle<c_{1+n_k}\},\quad \Omega_n=\bigcap_{k=1}^l\{x;c_{n_k}< \langle x,\xi_k\rangle
<c_{1+n_k}\}.
\eneqn
Then  $Z_n=\ol{\Omega_n+\gamma^{a}}\cap(\Omega_n+\gamma)$ and $Z_n$ is $\gamma$-locally closed. 

Since $F_{\Omega_n}$ is constant, $F_{Z_n}$ is constant by Theorem~\ref{th:locctonflat}. 

Now we have $\BBV=\bigsqcup_{n\in [0,N-1]^l}Z_n$. 
Set $A=\set{n\in [0,N-1]^l}{F\vert_{Z_n}\not\simeq0}$.
Then $\supp(F)=\bigcup_{n\in A}\ol{Z_n}$.
\begin{lemma}
There exists $n\in A$ such that $Z_n$ is open in $\supp(F)$. 
\end{lemma}
\begin{proof}[Proof of the lemma]
We order the set of $n$'s by $n\leq n'$ if $n_j\leq n'_j$ for all $j\in[1,\dots,l]$. 
Let $n$ be a minimal element of $A$. Then $Z_n$ is open in $\supp(F)$. Indeed, 
\eqn
Z_n=\supp(F)\cap Z_n&= &\supp(F)\cap\bigcap_k\{x;c_{n_k}\leq \langle x,\xi_k\rangle<c_{n_k+1}\}\\
&=&\supp(F)\cap\bigcap_k\{x; \langle x,\xi_k\rangle<c_{n_k+1}\}.
\eneqn
Note that the last equality is true since otherwise there exists 
$n'< n$ in $A$ such that
$\supp(F)\cap Z_{n'}\not=\emptyset$,
and $n$ would not be minimal. 
\end{proof}
Now we can complete the proof of Theorem~\ref{th:genergammatri}. 

Let us take $n\in A$ such that $Z_n\subset \supp(F)$ is open in $\supp(F)$.
Then we have an exact sequence
\eqn
&&0\to\cor_{Z_n}\tens F(Z_n)\to F\to F''\to0
\eneqn
and $\supp F''\subset \supp(F)\setminus Z_n$. 
Then the proof goes by induction on $\#A$. 
\end{proof}

\vspace*{1cm}
\noindent
\parbox[t]{21em}
{\scriptsize{
\noindent
Masaki Kashiwara\\
 Kyoto University \\
Research Institute for Mathematical Sciences\\
Kyoto, 606--8502, Japan\\
and\\
Department of Mathematical Sciences and School of Mathematics,\\
Korean Institute for Advanced Studies, \\
Seoul 130-722, Korea

\medskip\noindent
Pierre Schapira\\
Sorbonne Universit{\'e}\\
Institut de Math{\'e}matiques de Jussieu\\
e-mail: pierre.schapira@imj-prg.fr\\
http://webusers.imj-prg.fr/\textasciitilde pierre.schapira/
}}

\end{document}